\documentclass[11pt]{amsart}
\usepackage{amssymb}
\usepackage{amsthm}
\usepackage{epsfig}
\usepackage{color}
\usepackage{verbatim}
\usepackage{appendix}

\usepackage{amsmath}
\usepackage{hyperref}
\usepackage[active]{srcltx}
\usepackage{xypic}
\usepackage{amscd}
\usepackage{color}
\newfont{\sheaf}{eusm10 scaled\magstep1}

\newtheorem{thm}{Theorem}[section]

\newtheorem{lemma}[thm]{Lemma}
\newtheorem{prop}[thm]{Proposition}

\newtheorem{proposition}[thm]{Proposition}

\theoremstyle{definition}
\newtheorem{remark}[thm]{Remark}

\newtheorem{definition}[thm]{Definition}
 \newtheorem{notation}[thm]{Notation}

\DeclareMathOperator{\PGL}{PGL}

\DeclareMathOperator{\Aut}{Aut}

\DeclareMathOperator{\Pic}{Pic}

\def\c1{\operatorname{c_1}}
\def\c2{\operatorname{c_2}}

\def\Sym{\operatorname{Sym}}

\def\rk{\operatorname{rk}}

\def\n{\mathbf{N}}

\def\c{\mathfrak{P}}
\def\CC{{\mathbb C}}
\def\ZZ{{\mathbb Z}}

\def\PP{{\mathbb P}}
\def\A{{\mathcal A}}

\def\R{{\mathcal R}}

\def\M{{\mathcal M}}
\def\N{{\mathcal N}}
\def\O{{\mathcal O}}
\def\I{{\mathcal I}}
\def\D{{\mathcal D}}

\def\F{{\mathcal F}}

\def\P{{\mathcal P}}

\def\x{\times}                   
\def\cong{\simeq}

\def\+{\oplus}                   
\def\*{\otimes}                  

\def\mod{\operatorname{mod}}
\def\Aut{\operatorname{Aut}}

\def\Pic{\operatorname{Pic}}

\def\det{\operatorname{det}}

\begin{document}

\title[Moduli of non-standard Nikulin surfaces in low genus]{Moduli of non-standard Nikulin surfaces in low genus}

\author[A.~L.~Knutsen]{Andreas Leopold Knutsen}
\address{A.~L.~Knutsen, Department of Mathematics, University of Bergen,
Postboks 7800,
5020 Bergen, Norway}
\email{andreas.knutsen@math.uib.no}

\author[M.~Lelli-Chiesa]{Margherita Lelli-Chiesa}
\address{M.~ Lelli-Chiesa, Dipartimento di Ingegneria e Scienze dell'Informazione  e Matematica, Universit{\`a} degli Studi dell'Aquila, Via Vetoio, localit{\`a} Coppito, 67100 L'Aquila, Italy}
\email{margherita.lellichiesa@univaq.it} 

\author[A.~Verra]{Alessandro Verra}
\address{A.~ Verra, 
Dipartimento di Matematica,
Universit{\`a} Roma Tre,
Largo San Leonardo Murialdo,  
00146 Roma,  Italy} \email{verra@mat.uniroma3.it}

\begin{abstract}
Primitively polarized genus $g$ Nikulin surfaces $(S,M,H)$ are of two types,
that we call standard and non-standard depending on whether the lattice
embedding $\ZZ[H] \+_{\perp} \mathbf{N} \subset \Pic S$ is primitive. Here
$H$ is the genus $g$ polarization and $\mathbf{N}$ is the Nikulin lattice.
We  concentrate on the non-standard case, which only occurs in odd
genus. In particular, we study the birational geometry of the moduli space
of non-standard Nikulin surfaces of genus $g$ and prove its rationality for
$g=7,11$ and the existence of a rational double cover of it when $g=9$.
Furthermore, if $(S,M,H)$ is general in the above moduli space and
$(C,M|_C)$ is a general Prym curve in $|H|$,  we determine the dimension of
the family of non-standard Nikulin surfaces of genus $g$ containing $(C,
M|_C)$ for $3\leq g\leq 11$; this completes the study of the Prym-Nikulin
map initiated in \cite{KLV}.
\end{abstract}

\maketitle

\section{Introduction} \label{sec:intro}
A Nikulin surface is a $K3$ surface endowed with a non-trivial double cover branched along eight disjoint rational curves. Nikulin surfaces have attracted a lot of attention in recent time because of their relevance in the study of both the moduli \cite{FV} and the syzygies \cite{FK1, FK2} of Prym canonical curves. It is imperative to recall the lattice theoretical proof by Sarti and van Geemen \cite{vGS} of the existence of exactly two types of polarized Nikulin surfaces, that we call {\it standard} and {\it non-standard} (cf. \S\ref{sec:def}), the latter occurring only in odd genera. There are coarse moduli spaces $\F_g^{\n,s}$ and $\F_g^{\n,ns}$ parametrizing genus $g$ primitively polarized Nikulin surfaces of standard and non-standard type, respectively; more precisely, a point of $\F_g^{\n,s}$ (resp., $\F_g^{\n,ns}$) represents a triple $(S,M,H)$, where $S$ is a standard (resp., non-standard) Nikulin surface, $H\in \Pic S$ is a genus $g$ primitive polarization and the line bundle $M\in \Pic S$ defines the double cover branched along eight disjoint rational curves. Both $\F_g^{\n,s}$ and $\F_g^{\n,ns}$ are irreducible of dimension $11$, cf. \cite[\S 3]{Do}, \cite[Prop. 2.3]{vGS}. 

Up to now, only the moduli spaces $\F_g^{\n,s}$ have been extensively studied, while non-standard Nikulin surfaces have not been adequately considered. This paper aims to (partially) fill this gap.    
We concentrate on the $\mathbb P^g$-bundle over $\F_g^{\n,ns}$ parametrizing pairs $((S,M,H),C)$ such that $(S,M,H) \in \F_g^{\n,ns}$ and $C\in |H|$. Let 
 $\P_g^{\n,ns}$ be the open set of pairs such that $C$ is smooth and let $\mathcal R_g$ be the moduli space of Prym curves; we look at the diagram 
\begin{equation}\label{nikulin}
\xymatrix{ 
&\P_g^{\n,ns} \ar[ld]_{q_g^{\n,ns}}  \ar[d]_{\chi_g^{ns}}    \ar[rd]^{m_g^{\n,ns}} &\\
\F_g^{\n,ns}&   \R_g \ar[r]    &\M_g\,,
}
\end{equation}
 whose arrows can be described as follows: $q_g^{\n,ns}$ and  $m_g^{\n,ns}$ are the obvious forgetful maps. Moreover, the {\it Prym-Nikulin map} $\chi_g^{ns}$ sends $((S,M,H),C)$ to the Prym curve $(C, M\otimes \O_C)$. In particular, $m_g^{\n,ns}$ is just the composition of  $\chi_g^{ns}$ and the forgetful map $\mathcal R_g \to \mathcal M_g$.
 
The main difference between the standard and non-standard case is that a general hyperplane section of a general polarized Nikulin surface of standard type is Brill-Noether general, while curves lying on non-standard Nikulin surfaces carry two unexpected theta-characteristics (cf. Prop. \ref{prop:extheta}) that make them special in moduli. A first consequence is that the maps $m_g^{\n,ns}$ and $\chi_g^{ns}$ can never be dominant. Furthermore, a heuristic count suggests that they cannot be generically finite for $g\leq 11$, cf. Remark \ref{rem:heur}. In \cite{KLV} we proved that the map $\chi_g^{ns}$ is birational onto its image for (odd) genus $g \geq 13$, and the behaviour of the analogous map in the standard case was completely described. 
In this paper, we complete the picture by showing that:

\begin{thm} \label{intro-thm:main}
The map $\chi_g^{ns}$ has generically 
\begin{itemize}
\item $9$-dimensional fibers for $g=3$;
\item $6$-dimensional fibers for $g=5$;
\item $4$-dimensional fibers for $g=7$;
\item $2$-dimensional fibers for $g=9$;
\item $1$-dimensional fibers for $g=11$.
\end{itemize}
\end{thm}

As already mentioned, hyperplane sections of non-standard Nikulin surfaces have some peculiar and compelling properties, that we now describe in more detail. A general genus $g$  polarized non-standard Nikulin surface $(S,M,H)$ carries two line bundles $R,R'$ such that $H(-M)\simeq R\otimes R'$. The restrictions of  $R$ and $R'$ to a general hyperplane section $C\in |H|$ are two theta-characteristics with positive dimensional spaces of global sections. For (odd) genus $g\geq 5$ both $h^0(\O_C(R))\geq 2$ and $h^0(\O_C(R'))\geq 2$ and hence the theta divisor of the Jacobian of $C$ has two singular points of given multiplicity. We precisely describe the images of $m_g^{\n,ns}$ for $g=3$ and $5$, cf. Theorems \ref{thm:g3} and \ref{thm:g5}:
\begin{itemize}
\item[--] the image of $m_3^{\n,ns}$ is the hyperelliptic locus in $\M_3$;
\item[--] the  image of $m_5^{\n,ns}$  coincides with the locus of curves in $\M_5$ possessing two autoresidual $g^1_4$; in particular, this locus is irreducible. 
\end{itemize}

For $g\geq 7$ the situation becomes more intricate  and the birational geometry of the moduli space $\F_g^{\n,ns}$ is worth investigating. We prove:

\begin{thm}\label{bir}
The moduli space $\F_g^{\n,ns}$ of non-standard Nikulin surfaces of genus $g$ is:
\begin{itemize}
\item rational for $g=7$ and $g=11$;
\item unirational with a rational double cover for $g=9$.
\end{itemize}
\end{thm}

The proof of both Theorems \ref{bir} and \ref{intro-thm:main} for $g\geq 7$ is given in \S \ref{sec:7}-\ref{sec:11} and relies on the description of nice projective models of non-standard Nikulin surfaces $(S,M,H)$ in low genus. Set $r:=h^0(R)-1$ and $r':=h^0(R')-1$. As already remarked by Garbagnati and Sarti in \cite{GS}, the line bundles $R$ and $R'$ enable to realize $S$ as a subvariety of the intersection of the Segre variety $\mathbb P^{r'}\times \mathbb P^{r}\subset \mathbb P^{rr'+r+r'}$ with a linear space of dimension $g-2$, namely, 
$\PP(H^0(S,H(-M))^\vee)$. We are able to detect some geometric conditions that are also sufficient for such a subvariety of $(\PP^{r'}\times \PP^{r})\cap\PP^{g-2}$ to be a Nikulin surface of non-standard type. 

For instance, a general non-standard Nikulin surface of genus $7$ is a divisor of bidegree $(2,3)$ in $\PP^1\times \PP^2$, cf. \cite[\S 4.8]{GS}. Furthermore, a $K3$ surface in $|\O_{\PP^1\times\PP^2}(2,3)|$ is a Nikulin surface of non-standard type if and only if it contains two conics $A_1$ and $A_2$ that are contracted by the first projection $\PP^1\times \PP^2\to \PP^1$ and are mapped to the same plane conic by the second projection $\PP^1\times \PP^2\to \PP^2$. 

Analogously, a general surface in $\F_9^{\n,ns}$ is a quadratic section of a Del Pezzo threefold $T:=(\PP^2\times \PP^2)\cap \PP^7\subset\PP^8$, cf. \cite[\S 4.9]{GS}. Moreover, an element in $|\O_T(2,2)|$ is a non-standard Nikulin surface if and only if it contains two sets of four  lines that are contracted by the first and second projection, respectively. 

As regards genus $11$, a general surface $S$ in $\F_{11}^{\n,ns}$ defines a divisor of type $(1,2)$ in the threefold $T':=(\PP^2\times\PP^3)\cap\PP^9\subset\PP^{11}$. The projection $T'\to\mathbb{P}^3$ realizes $T'$ as the blow-up of $\mathbb{P}^3$ along a rational normal cubic curve $\gamma$ and we denote by $P_\gamma$ the exceptional divisor. The surface $S$ intersects $P_\gamma$ along a rational quintic curve $\Gamma\subset T'\subset \mathbb P^9$ and in fact we show that the containment of $\Gamma$ is a necessary and sufficient condition for a surface in $|\O_{T'}(1,2)|$ to be a non-standard Nikulin surface of genus $11$. The  rationality results in Theorem \ref{bir} will follow from these characterizations.

Concerning the fibers of the moduli map $\chi_g^{ns}$, the case of genus $7$ has some special features. Let $C\subset (\mathbb P^{r}\times \mathbb P^{r'})\cap\PP^{g-2}$ be a general genus $g$ Nikulin section in the non-standard case. In genus $9$ a general quadratic section of the threefold $T$ containing $C$ is a non-standard Nikulin surface; the same holds in genus $11$ if one considers in the threefold $T'$ a general divisor of type $(1,2)$ through $C$. 
The situation in genus $7$ is divergent: a general $K3$ surface in the linear system $|\I_{C/\PP^1\times\PP^2}(2,3)|$ is not a Nikulin surface. This difference depends on the fact that, contrary to what happens for $g=9,11$, in genus $7$ the embedded curve $C\subset \PP^{g-2}$ is not quadratically normal. As a relevant consequence,  the image of $\chi_7^{ns}$ lies in the ramification locus of the Prym map $\R_7\to \A_6$, cf. Remark \ref{rem:7ram} and \cite{Be}. This suggests an interesting behaviour of Nikulin sections with respect to their Prym varieties. In the standard case this phenomenon was already pointed out in \cite{FV}, where the image of  $\chi_6^{\n,s}$ is identified with the ramification locus of the Prym map $\R_6\to \A_5$, but was still unknown in the non-standard case.

\vspace{0.3cm}
\noindent{\it Acknowledgements.}
The first author has been partially supported by grant n. 261756 of
the Research Council of Norway. The second and third named authors were supported
by the Italian PRIN-2015 project ``Geometry of Algebraic varieties'' and the third by GNSAGA.

\section{Nikulin surfaces of non-standard type and Segre varieties} \label{sec:def}

We recall some basic definitions and properties.
\begin{definition} \label{def:Nik}
  A polarized  {\it Nikulin} surface of genus $g \geq 2$ is a triple $(S,M,H)$ such that 
$S$ is a smooth $K3$ surface, $\O_S(M),H \in \Pic S$ and the following conditions are satisfied:
\begin{itemize}
\item $S$ contains $8$ mutually disjoint rational curves $N_1,\ldots,N_8$ such that $$ N_1+\cdots+N_8 \sim 2M. $$
\item $H$ is nef, $H^2=2(g-1)$ and  $H \cdot M=0$.
\end{itemize}

We say that $(S,M,H)$ is {\it primitively polarized} if in addition $H$ is primitive in $\Pic S$.
\end{definition}

\begin{definition} \label{def:Nik2}
Let $(S,M,H)$ be a Nikulin surface of genus $g$. Its {\it Nikulin lattice} $\mathbf{N}=\mathbf{N}(S,M)$ is the rank $8$ sublattice of $\Pic S$ generated by $N_1,\ldots,N_8$ and $M$.

One also defines the rank $9$ lattice
\[ \Lambda=\Lambda(S,M,H):= \ZZ[H] \+_{\perp} \mathbf{N} \subset \Pic S.\]

If the embedding $\Lambda \subset \Pic S$ is primitive, we call $(S,M,H)$  a {\it Nikulin surface of standard type}, else we call it a {\it Nikulin surface of non-standard type}.
\end{definition} 

There are coarse moduli spaces 
$\F_g^{\n,s}$ (respectively, $\F_g^{\n,ns}$) parametrizing polarized Nikulin surfaces of genus $g$ of standard (resp., non-standard) type. Both $\F_g^{\n,s}$ and $\F_g^{\n,ns}$ are irreducible of dimension $11$ and their very general members have Picard number nine, cf. \cite[\S 3]{Do}, \cite[Prop. 2.3]{vGS}. 
By \cite[Prop. 2.2]{vGS}, if $(S,M,H)$ is a non-standard Nikulin surface of genus $g$, then  $g$ is odd and the embedding $\Lambda \subset \Pic S$ has index 
two. More precisely (cf. \cite[Prop. 2.1 and Cor. 2.1]{GS}), possibly after renumbering the curves $N_i$, there are  $R, R' \in \Pic S$ such that 
\begin{itemize}
\item $H-N_1-N_2-N_3-N_4 \sim 2R$ and $H-N_5-N_6-N_7-N_8 \sim 2R'$ if $g \equiv 1 \; \mod 4$; 
\item $H-N_1-N_2 \sim 2R$ and $H-N_3-\cdots-N_8 \sim 2R'$ if $g \equiv 3 \; \mod 4$.
\end{itemize}
Moreover, when $\rk \Pic S=9$, then $\Pic S \cong \ZZ[R] \+ \mathbf{N}$ by 
\cite[Prop. 2.1 and Cor. 2.1]{GS}.
We also need to define the line bundle $L:=H-M$, which satisfies $L^2=2(g-3)$ and $L \cdot N_i=1$ for $i=1,\ldots,8$. 

We henceforth concentrate on Nikulin surfaces of non-standard type.

First of all we show that
hyperplane sections of non-standard Nikulin surfaces 
are rather special.

\begin{proposition} \label{prop:extheta}
 Let $(S,M,H)$ be a general non-standard Nikulin surface of genus $g \equiv 1 \; \mod 4$ (respectively, $g \equiv 3 \; \mod 4$) and let $L$, $R$ and $R'$ be as above. Then
  \begin{itemize}
  \item[(i)] $R$ and $R'$ are globally generated with $h^1(R)=h^1(R')=0$ if $g \geq 5$;
  \item[(ii)] $h^0(L)=g-1$ and $L$ is very ample if $g \geq 7$ and is ample and globally generated defining a degree two morphism onto $\PP^1 \x \PP^1\subset \PP^3$ if $g=5$;
\item[(iii)]if $g\geq 5$, then for any smooth curve $C$ in $|H|$, the line bundles $\O_C(R)$  and $\O_C(R')$ are theta-characteristics satisfying $h^0(S,R)=h^0(C,\O_C(R))=
(g+3)/4$ (resp., $(g+5)/4$) and $h^0(S,R')=h^0(C,\O_C(R'))=
(g+3)/4$ (resp., $(g+1)/4$).
  \end{itemize}
\end{proposition}

\begin{proof}
  Since all properties are open in the moduli space, one may prove them for a non-standard Nikulin surface with $\rk \Pic S=9$.  
Then (i) is
proved in \cite[Prop. 3.5(2)]{GS} (recalling that a linear system on a $K3$ surface without base components is base point free) and (ii) in  \cite[Prop. 3.2 and Lemma 3.1]{GS}, using the classical numerical criteria of Saint-Donat 
\cite{SD}. As $L^2=2(g-3)>0$, we have $h^1(L)=h^2(L)=0$, whence $h^0(L) =g-1$ by Riemann-Roch.

To prove (iii) we note that $R-H \sim -(R+N_1+N_2+N_3+N_4)$ (resp., $-(R+N_1+N_2)$). Thus, $h^0(R-H)=0$. Moreover, the linear system $|R|$ contains irreducible members thanks to (i), and hence $|H-R|$ contains a divisor $D$ that is the union of an irreducible element in $|R|$ and four rational irreducible tails. In particular, one has $h^0(\O_D)=1$ and thus $h^1(R-H)=0$. The standard restriction sequence yields 
$h^0(S,R)=h^0(C,\O_C(R))=\frac{1}{2}R^2+2$ by Riemann-Roch and (i). The rest then follows from an easy computation and the same argument applies to $R'$.
\end{proof}

\begin{remark}
In the embedding $S \subset \PP^{g-2}$ defined by $|L|$, any smooth $C$ in $|H|$ is mapped to a Prym-canonical curve, as $L|_C \cong \omega_C \* \O_C(M)$ satisfies $L|_C^{\otimes 2}\simeq \omega_C ^{\otimes 2}$, and all $N_1,\ldots, N_8$ are mapped to lines. 
\end{remark}

From now on, we will set $r:=h^0(S,R)-1$ and $r':=h^0(S,R')-1$. By Proposition \ref{prop:extheta},  as soon as $g\geq 7$, the two linear systems $|R'|$ and $|R|$ (and their restrictions to $C$) define an embedding 
\begin{equation}\label{segre1}
C\subset S \subset \PP^{r'} \x \PP^r \subset \PP^{rr'+r+r'},
\end{equation}
where the second inclusion is the Segre embedding.

\begin{notation} \label{not:1}
  We let $p:\PP^{r'} \x \PP^r \to \PP^r$ and $p': \PP^{r'} \x \PP^{r} \to \PP^{r'}$  be the two projection maps. For any subvariety $X \subset \PP^{r'} \x \PP^r$, we denote by $p'_X$ and $p_X$ the restrictions to $X$ of $p'$ and $p$, respectively. In particular, $p'_S$ and $p_S$ are the maps defined by $|R'|$ and $|R|$, respectively. 

We use the standard notation $\O_{ \PP^{r'} \x \PP^r}(a,b):= {p'}^*\O_{\PP^{r'}}(a)\otimes{p}^*\O_{\PP^r}(b)$, and for any  subvariety $X \subset \PP^{r'} \x \PP^r$, we set $\O_X(a,b) \cong \O_{ \PP^{r'} \x \PP^r}(a,b)|_X$ and refer to elements in the corresponding linear systems as {\it divisors of bidegree $(a,b)$ on $X$}. 
\end{notation}

\begin{definition} \label{def:verthor}
  We say that a curve in  $\PP^{r'} \x \PP^r$  
is {\it vertical} if it is contracted by $p'$ and
{\it horizontal} if it is contracted by $p$. 
\end{definition}

Any line $\ell$ in $\PP^{r'} \x \PP^r$ is either vertical or horizontal. 
If $g \equiv 1 \; \mod 4$, then
\begin{itemize} 
\item $N_1, \ldots, N_4$ are vertical, as $N_1 \cdot R' = \dots =N_4 \cdot R' = 0$;
\item $N_5, \ldots, N_8$ are horizontal, as $N_5\cdot R = \dots = N_8 \cdot R = 0$.
\end{itemize}
If instead $g \equiv 3 \; \mod 4$, then
\begin{itemize} 
\item $N_1, N_2$ are vertical, as $N_1 \cdot R' = N_2 \cdot R' = 0$;
\item $N_3, \ldots, N_8$ are horizontal, as $N_3\cdot R = \dots = N_8 \cdot R = 0$.
\end{itemize}

We will make use of the following:

\begin{lemma}\label{mult}
Let $(S,M,H)$ be a general non-standard Nikulin surface of genus $g$ with $5 \leq g \leq 15$ and let $L$, $R$ and $R'$ be as above.
Then the multiplication map
$$
\mu_{R,R'}:H^0(S,R)\otimes H^0(S,R') \longrightarrow H^0(S,L)
$$
is surjective. Furthermore, it is isomorphic to the multiplication map
$$
\mu_{\eta,\eta'}:H^0(C,\eta)\otimes H^0(C,\eta') \longrightarrow H^0(C,\eta\otimes \eta'),
$$
where $C$ is any smooth irreducible curve in $|H|$ and $\eta$ and $\eta'$ are the restrictions to $C$ of the line bundles $R$ and $R'$, respectively.
\end{lemma}

\begin{proof}
The properties are open in the moduli space, so we may assume that $\rk \Pic S=9$.
The surjectivity of $\mu_{R,R'}$ follows from the generalization by Mumford of a theorem of Castelnuovo, cf. \cite[Thm. 2, p.~41]{Mum} (recalling that the assumption on ampleness is unnecessary) once we check that $h^1(R-R')=0$
and $h^2(R-2R')=0$. 

We have $2(R-R') \sim -N_1-\cdots-N_4+N_5+\cdots+N_8$ (resp., $-N_1-N_2+N_3+\cdots+N_8$)
if $g \equiv 1 \; \mod 4$ (resp., $g \equiv 3 \; \mod 4$). Hence, $h^0(R-R')=h^0(R'-R)=0$. As $(R-R')^2=-4$, one has $h^1(R-R')=0$. 

We next prove that $h^2(R-2R')=h^0(2R'-R)=0$. We treat the case $g \equiv 1 \; \mod 4$, leaving the other case to the reader, as it is very similar. 

We
have $2R'-R \sim R+N_1+\cdots+N_4-N_5-\cdots-N_8$.
Since $N_i \cdot (2R'-R) =-1$ for $i \in \{1,2,3,4\}$, we have
$h^0(2R'-R)= h^0(R-N_5-\cdots-N_8)$. 
The assumption on $g$ implies that $(R-N_5-\cdots-N_8)^2 \leq -4$. Hence, if $R-N_5-\cdots-N_8$ is effective, it has nonvanishing $h^1$, which by Ramanujam's vanishing theorem \cite[Lemma 3]{ram} implies that 
it is not $1$-connected. Hence, there is an effective nontrivial decomposition $R-N_5-\cdots-N_8 \sim A+B$ such that $A \cdot B \leq 0$. 
Since $\Pic S \cong \ZZ[R] \+ \mathbf{N}$ by 
\cite[Prop. 2.1 and Cor. 2.1]{GS}, we may write
\[ A \sim \alpha R +\frac{1}{2}\sum_{i=1}^8 \alpha_iN_i \; \; \mbox{and} \; \; 
  B \sim \beta R +\frac{1}{2}\sum_{i=1}^8 \beta_iN_i,\]
for integers $\alpha,\beta, \alpha_i,\beta_i$ satisfying
\begin{equation}
  \label{eq:sati}
 \alpha+\beta=1, \; \; \alpha_i+\beta_i=0 \; \mbox{if} \; i \in \{1,2,3,4\} \; \; \mbox{and} \; \;  \alpha_i+\beta_i=-2 \; \mbox{if} \; i \in \{5,6,7,8\}.
\end{equation}
Effectivity requires that $\alpha \geq 0$ and $\beta \geq 0$, so that we can without loss of generality assume $\alpha=1$ and $\beta=0$. Therefore,
$B \sim  \frac{1}{2}\sum_{i=1}^8 \beta_iN_i$ and effectivity requires that all $\beta_i \geq 0$ and all $\beta_i$ are even. Write $\beta_i=2\gamma_i$ for integers $\gamma_i \geq 0$. Then $\alpha_i=-2\gamma_i$ if  $i \in \{1,2,3,4\}$ and
$\alpha_i=-2(\gamma_i+1)$
if $i \in \{5,6,7,8\}$, so that $A \sim R-\sum_{i=1}^4\gamma_i N_i-\sum_{i=5}^8(\gamma_i+1)N_i$. Therefore,
\begin{eqnarray*} A \cdot B & = & \left(R-\sum_{i=1}^4\gamma_i N_i-\sum_{i=5}^8(\gamma_i+1)N_i\right)\cdot \sum_{i=1}^8 \gamma_i N_i \\
& = & \sum_{i=1}^4\gamma_i (2\gamma_i+1) + 2\sum_{i=5}^8 \gamma_i(\gamma_i+1).
\end{eqnarray*}
Since at least one of the $\gamma_i$ is strictly positive, we see that we get $A \cdot B \geq 3$, a contradiction.

As concerns the second statement, it is enough to remark that the line bundles $L-H$, $R-H$ and $R'-H$ all have vanishing $h^0$ and $h^1$, which can be proved as in the last part of the proof of Proposition \ref{prop:extheta}.
\end{proof}

As a consequence, for $g\geq 7$ the embeddings $C\subset S\subset \mathbb{P}^{rr'+r+r'}$ in \eqref{segre1} factor through the embedding $S\subset \mathbb{P}^{g-2}$ defined by $|L|$, and hence:
\begin{equation}\label{segre}
C\subset S \subset ( \PP^{r'} \x \PP^r)\cap \PP^{g-2} \subset \PP^{rr'+r+r'}.
\end{equation}

\begin{remark} \label{rem:trans}
It is not a priori obvious that the intersection $( \PP^{r'} \x \PP^r)\cap \PP^{g-2}$ is transversal. However, since $\F_g^{\n,ns}$ is irreducible, as soon as one shows the existence of a  
non-standard Nikulin surface of genus $g$ in some transversal intersection 
$( \PP^{r'} \x \PP^r)\cap \PP^{g-2}$, one gets the transversality statement for a general Nikulin surface in  $\F_g^{\n,ns}$.
\end{remark}

The next two results prove Theorem \ref{intro-thm:main} in genera $3$ and $5$.

\begin{thm} \label{thm:g3} 
The image of $m_3^{\n,ns}$ coincides with the hyperelliptic locus in $\M_3$. In particular, a
 general fiber of  $m_3^{\n,ns}$ has dimension $9$. 
\end{thm}

\begin{proof}
Let $(S,M,H)\in \F_3^{\n,ns}$ and $C\in |H|$ be general. The restriction of the line bundle $R\in \Pic S$ to $C$ is a $g^1_2$; in particular, the canonical map of $C$ is a double cover of a plane conic $C_K$ branched along $8$ points. Furthermore, the linear system $|H|$ on $S$ defines a double cover $\varphi_H:S\to X\subset \PP^3$
of a cone $X$ in $\PP^3$ branched along a plane conic $C_2$ that is the image of the unique curve in $|R'|$, and a sextic $C_6$ that is the image of an irreducible curve in the linear system $|H+R'|$ (cf. \cite[4.3]{GS}). Note that $C_2\cdot C_K=2$, $C_6\cdot C_K=6$ and that $C_2$ and $C_6$  meet at the six points in $X$ that are images of the curves $N_3,\ldots, N_8$. Furthermore, $\varphi_H$ factors through 
\begin{equation}\label{factor}
S\stackrel{c}{\longrightarrow} \overline{S}\stackrel{\pi}{\longrightarrow}\mathbb{F}_2\stackrel{\phi}{\longrightarrow} X,
\end{equation}
where $\mathbb{F}_2$ is the second Hirzebruch surface (with a section $C_0$ such that $C_0^2=-2$ and class fiber denoted by $f$), the map $\phi$ is induced by the linear system $|C_0+2f|$ on $\mathbb{F}_2$, the map $\pi$ is a double cover branched along the inverse image of $C_2$ and $C_6$, while $c$ is the contraction of $N_3,\ldots,N_8$. Note that $(\pi\circ c)^{-1}(C_0)=N_1\cup N_2$, and $\phi^*C_2\in |C_0+2f|$ while $\phi^*C_6\in |3C_0+6f|$.

 It is not difficult to show that the desingularization $S$ of any double cover $\overline{S}$ of $\mathbb{F}_2$ branched along the union of a smooth irreducible curve  $C_2\in |C_0+2f|$ and a smooth irreducible curve $C_6\in |3C_0+6f|$ is a Nikulin surface. Indeed, $S$ is a $K3$ surface by, e.g., \cite[Thm. 2.2]{Re}; furthermore, $S$ has eight disjoint rational curves, two of which mapping to the section $C_0$ (call them $N_1$ and $N_2$) and six arising as exceptional divisors of the desingularization of $\overline{S}$ (call them $N_3,\ldots,N_8$), which has six double points at the inverse images of $C_2\cap C_6$. The line bundle $H\in \Pic S$ obtained as pullback of $C_0+2f$ is a genus $3$ polarization. We denote by $R\in\Pic S$ the pullback of $f$, and by $R'\in \Pic S$ the line bundle with a section vanishing at the strict transform in $S$ of  the ramification curve $\pi^{-1}(C_2)\subset \overline S$. In particular, we have 
 \begin{equation}\label{ns}
 H-N_1-N_2\sim (\pi\circ c)^*(C_0+2f)-(\pi\circ c)^*C_0\sim 2R.
 \end{equation}Setting $M:=H-R-R'$, one easily checks that $ N_1+\cdots+N_8 \sim 2M$ and hence $(S,M,H)$ is a genus $3$ Nikulin surface of  non-standard type by \eqref{ns}; it depends on  $\dim |C_0+2f|+\dim |3C_0+6f|-\dim \Aut(\mathbb{F}_2)=3+15-7=11$ moduli.

 We use this in order to prove that a general hyperelliptic curve of genus $3$ lies on a Nikulin surface of non-standard type. Let $C$ be a general hyperelliptic curve of genus $3$ and let $C_K\subset \PP^2$ be the canonical image of $C$, which is a smooth plane conic. We denote by $x_1,\ldots, x_8$ the image in $C_K$ of the eight Weierstrass points on $C$ and by $X$ the cone in $\mathbb{P}^3$ over $C_K$.  The desingularization of $X$ is then isomorphic to $\mathbb{F}_2$. By abuse of notation, we still denote by the same name the inverse images in $\mathbb{F}_2$ of the curve $C_K$ and the points $x_1,\ldots, x_8$. It is then enough to remark that both the linear systems $|(C_0+2f)\otimes \I_{x_1+x_2}|$ and $|(3C_0+6f)\otimes \I_{x_3+\cdots+x_8}|$ are nonempty and contain smooth members. 
\end{proof}

\begin{thm} \label{thm:g5} 
The map $m_5^{\n,ns}$ has generically $6$-dimensional fibers and its image coincides with the locus of curves in $\M_5$ possessing two autoresidual $g^1_4$. In particular, this locus is irreducible of dimension $10$.
\end{thm}

\begin{proof}
By  \cite[4.6(b)]{GS}, the nodal model of a Nikulin surface $S$ of non-standard type and genus $5$  is the complete intersection in $\PP^5$ of three quadrics $Q_1,Q_2,Q_3$ such that $Q_3$ is smooth, while  $Q_1$ and $Q_2$ have rank $3$ and disjoint singular loci. 

Vice versa, we are going to show that the minimal desingularization of any complete intersection $\overline{S}=Q_1\cap Q_2\cap Q_3$ of three quadrics in $\PP^5$ with the above properties is automatically a Nikulin surface of non-standard type and genus $5$. For $i=1,2$, let $\pi_i$ be the plane vertex of $Q_i$. The plane $\pi_1$ (respectively, $\pi_2$)  intersects $S$ at four nodes $P_1,\ldots,P_4$ (resp., $P_5,\ldots P_8$).  Let $q:S\to \overline{S}$ be the minimal desingularization of $\overline{S}$ and let $N_i:=q^{-1}(P_i)$ for $1\leq i\leq 8$.  The line bundle $H:=q^*(\O_{\overline{S}}(1))$ is a genus 5 polarization on $S$. Up to a change of coordinates, the quadrics $Q_1$ and $Q_2$ have defining equations $z_0z_1-z_2^2=0$ and $z_3z_4-z_5^2=0$, respectively; hence, $\pi_1:z_0=z_1=z_2=0$ and $\pi_2:z_3=z_4=z_5=0$. The hyperplanes $z_0=0$ and $z_1=0$ generate a pencil of hyperplanes in $\PP^5$ all  passing  through the points $P_1,\ldots,P_4$ and cutting out on $\overline S$ a curve with multiplicity two; therefore, there exists a line bundle $R\in \Pic S$ such that $2R\sim H-N_1-N_2-N_3-N_4$. Analogously, one shows the existence of a line bundle $R'\in \Pic S$ such that $2R'\sim H-N_5-N_6-N_7-N_8$. Hence, $S$ is a Nikulin surface of non-standard type.

We are now ready to detect the image of $m_5^{\n,ns}$. First of all, it is straightforward that the line bundles $R$ and $R'$ on a genus $5$ Nikulin surface of non-standard type cut out two autoresidual $g^1_4$ on a general hyperplane section. The other way around, let us consider a genus $5$ curve $C$ possessing two autoresidual $g^1_4$; these determine two rank-$3$ quadrics $q_1$ and $q_2$ in $\PP^4$ containing the canonical image of $C$, cf. \cite[p.~208]{acgh}. Since any component of the locus in $\M_5$ of curves with two autoresidual $g^1_4$ has dimension at least $10$, we can assume $C$ not to be bielliptic; this ensures that the singular lines of $q_1$ and $q_2$ do not intersect, cf. \cite[ch. VI, F]{acgh}. Fix an embedding $\PP^4\subset\PP^5$ and let $Q_1$ (respectively, $Q_2$) be the cone over $q_1$ (resp. $q_2$) with vertex a point $P_1$ (resp., $P_2$) in  $\PP^5\setminus\PP^4$; then, both $Q_1$ and $Q_2$ are quadrics of rank $3$ and one can choose the points $P_1$ and $P_2$ so that their singular loci are disjoint.  It is easy to check that $h^0(\PP^5,\I_{C/\PP^5}(2))=9$ and a general quadric $Q_3$ containing $C$ is smooth since $C$ cannot be trigonal (cf. \cite[ch. VI, F]{acgh}); therefore, the surface $\overline{S}=Q_1\cap Q_2\cap Q_3$ is the nodal model of a Nikulin surface of non-standard type. Furthermore, the fiber of  $m_5^{\n,ns}$ over $[C]$ is  parametrized by $\PP(H^0(\PP^5,\I_{C/\PP^5}(2))/\langle Q_1,Q_2 \rangle)=\PP^6$.
\end{proof}

The rest of the paper will focus on the cases $g=7,9,11$. 

\begin{remark} \label{rem:heur}
 The following  heuristic count shows that the expected dimension of a general fiber of $\chi_g^{ns}$ for $g=7,9,11$ is the one obtained in Theorem \ref{intro-thm:main}. 

When $g=7$, a general hyperplane section $C$ 
carries two theta-characteristics with a space of global sections of dimension $3$ and $2$, respectively, by 
Proposition \ref{prop:extheta}. The moduli spaces of such curves have codimensions $3$ and $1$, respectively, in $\M_g$ or $\R_g$, by \cite{Te}, thus one expects the target of $\chi_7^{ns}$ to have dimension $18-3-1=14$ and the fibers to have dimension $11+7-14=4$. 

When $g=9$, a general hyperplane section $C$ 
carries two theta-characteristics with a $3$-dimensional space of global sections, by 
Proposition \ref{prop:extheta}. The moduli spaces of such curves have codimension $3$ in $\M_g$ or $\R_g$, by \cite{Te}, thus one expects the target of $\chi_9^{ns}$ to have dimension $24-3-3=18$ and the fibers to have dimension $11+9-18=2$. 

When $g=11$, a general hyperplane section $C$ 
carries two theta-characteristics with $4$ and $3$ sections, respectively, by 
Proposition \ref{prop:extheta}. The moduli spaces of such curves have codimensions $6$ and $3$, respectively, in $\M_g$ or $\R_g$, by \cite{Te}, thus one expects the target of $\chi_{11}^{ns}$ to have dimension $30-6-3=21$ and the fibers to have dimension $11+11-21=1$. 
\end{remark}

\section{The case of genus $7$} \label{sec:7}

Let $(S,M,H)$ be a general primitively polarized Nikulin surface of non-standard type of genus $7$.
Let $L=H-M$ and 
\[ R \sim \frac{1}{2} \left(H-N_1-N_2\right) \; \; \mbox{and} \; \; 
R' \sim L-R \sim \frac{1}{2} \left(H-N_3-\cdots-N_8\right) \]
be as in \S \ref{sec:def}. By Proposition \ref{prop:extheta}, the line bundle $L$ defines an embedding $S \subset \PP^5$ and the embeddings in \eqref{segre} are as follows: 
$$
S \subset \PP^1 \x \PP^2 \subset \PP^5.
$$
Here $|R| = | \mathcal O_S(0,1) |$ is a net  of genus $2$ curves  of degree $R \cdot L= 5$ and  $|R'| = |\mathcal O_S(1,0)|$ is a  pencil of elliptic curves of degree $R' \cdot L=3$.  
By the adjunction formula, $S \in |\mathcal O_{\PP^1 \x \PP^2}(2,3)|$, cf. \cite[\S 4.8]{GS}. 
We want to identify the locus in $|\mathcal O_{\PP^1 \x \PP^2}(2,3)|$ parametrizing Nikulin surfaces of non-standard type.  Since $R'\cdot N_1 = R'\cdot N_2 = 0$, two elements of $|R'|$ split  as $N_1 + A_1$ and $N_2 + A_2$.  In particular $A_1, A_2$ are two disjoint conics in the embedding $S \subset \PP^5$, mapped into conics in $\PP^2$ by $p$, as $R \cdot A_1=R \cdot A_2=2$. Furthermore,  one can prove that $A_1$ and $A_2$ are irreducible by specializing to the case where $\rk \Pic S=9$ and proceeding as in \cite[proof of Prop. 3.5(2)]{GS}. 

\begin{lemma} \label{lemma:uguali}
We have $ p(A_1) = p(A_2)$. 
\end{lemma}

\begin{proof} 
Since $N_j \cdot R'=1$ and $N_j \cdot N_1= N_j \cdot N_2=0$ for $j \geq 3$, we have $N_j \cdot A_1=N_j \cdot A_2=1$ for $j \geq 3$. It is then enough to note that the six points $z_j := p(N_j)$, $j = 3, \ldots, 8$, are distinct and belong to both the conics $p(A_1)$ and $p(A_2)$.
\end{proof}

We call $A_1$ and $A_2$ the {\it vertical conics of $S$}. 

Using the fact that $R' \sim N_1+A_1 \sim N_2+A_2 \sim \frac{1}{2}\left(H-N_3-\cdots-N_8\right)$, we obtain $2R'\sim \left(N_1+A_1\right) + \left(N_2+A_2\right)  \sim H-N_3-\cdots-N_8$, whence
\begin{equation}
  \label{eq:H-2M}
H \sim N_1+\cdots+N_8+A_1+A_2.
\end{equation}

\subsection{Rationality of $\F_7^{\n,ns}$} \label{ss:ratpar7}

Fix any smooth conic $A \subset \mathbb P^2$ and two disjoint vertical conics $A_1, A_2 \subset \mathbb P^1 \times \mathbb P^2$ such that $p(A_1) = p(A_2)=A$. 
The surface $\PP^1 \times A$ is of bidegree $(0,2)$ in $\PP^1 \x \PP^2$. 
Consider the inclusion
$$
| \I_{\mathbb P^1 \times A/\PP^1 \x \PP^2 }(2,3) | \subset | \I_{A_1 \cup A_2/\PP^1 \x \PP^2}(2,3)|.
$$

\begin{proposition} \label{prop:7-tuttenik}
A general member of $| \I_{A_1 \cup A_2/\PP^1 \x \PP^2}(2,3)|$ is smooth and
every smooth  $S \in | \I_{A_1 \cup A_2/\PP^1 \x \PP^2}(2,3)|$ is a non-standard  Nikulin surface of genus $7$ polarized by $\O_{S}(2,2)(-A_1-A_2)$. 

Moreover, $\dim | \I_{A_1 \cup A_2/\PP^1 \x \PP^2}(2,3)|=15$ and 
$\dim |\I_{\mathbb P^1 \times A/\PP^1 \x \PP^2 }(2,3)| =8$.
\end{proposition}

\begin{proof} 
The standard exact sequence 
\[ \xymatrix{0 \ar[r] &  \I_{\PP^1 \x A/\PP^1 \x \PP^2}(2,3) \ar[r] &  \I_{A_1 \cup A_2/\PP^1 \x \PP^2}(2,3) \ar[r] &  \I_{A_1 \cup A_2/\PP^1 \x A}(2,3) \ar[r] &  0} \]
along with the isomorphisms $\I_{\PP^1 \x A/\PP^1 \x \PP^2}\cong \O_{\PP^1 \x \PP^2}(0,-2)$ and 
$\I_{A_1 \cup A_2/\PP^1 \x A}(2,3) \cong \O_{\PP^1 \x \PP^1}(-2,0) \* \O_{\PP^1 \x \PP^1}(2,6) \cong \O_{\PP^1 \x \PP^1}(0,6)$ proves the dimensional statements, the global generation of 
$\I_{A_1 \cup A_2/\PP^1 \x \PP^2}(2,3)$ and the surjectivity of the  restriction map of linear systems
\[ \rho: |\I_{A_1 \cup A_2/\PP^1 \x \PP^2}(2,3)| \longrightarrow (A_1 \cup A_2) + 
|\O_{\PP^1 \x \PP^1}(0,6)|. \]
Hence, a general $S \in | \I_{A_1 \cup A_2/\PP^1 \x \PP^2}(2,3)|$ is smooth and 
\[
  S \cdot (\PP^1 \times A) = A_1 + A_2 + N_3 + \dots + N_8 \in |\O_{S}(0,2)|,
\]
with $N_3,\ldots,N_8$ disjoint horizontal lines. At the same time, $|\O_{S}(1,0)|$ is a pencil of elliptic curves of degree $3$ on $S$ such that $\O_{S}(1,0)\cdot A_i=0$ for $i=1,2$,  and hence contains two elements of the form $N_i+A_i$ with $N_i$ a line for $i=1,2$. Furthermore,  $N_1$ and $N_2$ are mutually disjoint, as well as disjoint from the other $N_j$ for $j=3,\ldots,8$. 
Note that the divisor $N_1+\cdots+N_8 \in |\O_{S}(2,2)(-2A_1-2A_2)|$ and thus is $2$-divisible in $\Pic S$.
It is now straightforward that $S$ satisfies the desired properties; in particular, \eqref {eq:H-2M} implies that $S$ is of non-standard type.
\end{proof}

Two smooth elements in $| \I_{A_1 \cup A_2/\PP^1 \x \PP^2}(2,3)|$ are isomorphic if and only if they are in the same orbit under the action of the stabilizer $G$ of $A_1 \cup A_2$ in $\Aut (\PP^1 \times A)$. The group $G$ is $4$-dimensional, since it is the product of
the stabilizer of two points in $\mathbb P^1$ and of the group $\Aut A$. Hence the 
quotient
$| \I_{A_1 \cup A_2/\PP^1 \x \PP^2}(2,3)| / G$ is $11$-dimensional and we have a birational map
\[ \xymatrix{|\I_{A_1 \cup A_2/\PP^1 \x \PP^2}(2,3)| / G \ar@{-->}[r] & \F_7^{\n,ns}.} \]

\begin{thm} \label{thm:rat7}
The moduli space $\F_7^{\n,ns}$ is rational. 
\end{thm}
\begin{proof} 
The blow-up of $\mathbb P^{15} := | \I_{A_1 \cup A_2/\PP^1 \x \PP^2}(2,3)|$ along $\mathbb P^8 := | \I_{\PP^1 \x A/\PP^1 \x \PP^2}(2,3)|$ is a $\PP^9$-bundle $\pi: \mathbb P \to \mathbb P^6$. Let
$o \in \mathbb P^6$, then $\pi^{-1}(o)$ is a 9-dimensional linear system generated by $\mathbb P^8$ and by an element $S \in \mathbb P^{15}$ not containing $\mathbb P^1 \times A$. It is useful to
remark that then the base locus of $\pi^{-1}(o)$ is $S \cdot (\mathbb P^1 \times A) = A_ 1 + A_2 + N_{o1} + \dots + N_{o6}$, where the last six summands are the 'horizontal' lines in the surface $S$.
Let $p'_A: \mathbb P^1 \times A \to A$ be the projection map. Since $N_{o1} + \dots + N_{o6} \in \vert \mathcal O_A(3) \vert$, this yields an immediate identification
$$
\mathbb P^6 := \vert \mathcal O_A(3) \vert  =  \vert \mathcal O_{\mathbb P^1}(6) \vert,
$$
under the linear isomorphism sending $o$ to $n := (p'_{A})_*(N_{o1}+ \dots+ N_{o6})$. Now it is clear that $G$ acts linearly on $\mathbb P$ and on $\mathbb P^6$. Furthermore, by Castelnuovo's criterion,
$\mathbb P^6 /G$ is a unirational surface, hence it is rational. To complete the proof it suffices to show that $\mathbb P / G$ is a $\mathbb P^9$-bundle over a nonempty open set of $\mathbb P/G$.
Let $U \subset \mathbb P^6$ be the open set of the degree six divisors $n \in \vert \mathcal O_A(3) \vert$ such that the stabilizer of $n$ in $\Aut A$ is trivial; this is nonempty since there are no non-trivial automorphisms of $\PP^1$ mapping a set of $6$ general points to itself. This immediately implies that, whenever $o \in U$, the stabilizer of $\pi^{-1}(o)$ in $G$ is trivial: otherwise $n$ would
be invariant under the action of some non trivial $\gamma \in G$. Let $\mathbb P_U$ be the restriction of $\mathbb P$ to $U$. Since the stabilizer of $\pi^{-1}(o)$ is trivial along $U$, it follows from
Kempf's descent lemma, cfr. \cite{DN},  that $\mathbb P_U$ descends to a $\mathbb P^9$-bundle $\mathbb P_U/G$ over $U/G$. This implies the statement.
\end{proof}

\subsection{The fibre of the Prym-Nikulin map $\chi_7^{ns}$} \label{ss:fibre7}

We start with a general point $(S, M,H)$ in $\F_7^{\n,ns}$ and a general smooth $C \in |H|$. We still denote by $A_1$ and $A_2$ the two vertical conics of $S$.

\begin{lemma} \label{lemma:7-cohprop}
We have 
\begin{itemize}
\item[(i)] $h^0(\I_{C/\PP^1 \x \PP^2}(2,2)) = h^1(\mathcal I_{C/\PP^1 \x \PP^2}(2,2)) = 1$,
\item[(ii)] $C$ is not quadratically normal in $\PP^5$,
\item[(iii)] $h^0(\I_{C/\PP^1 \x \PP^2}(2,3)) = 6$,
\item[(iv)] $h^0(\I_{C \cup A_1 \cup A_2/\PP^1 \x \PP^2}(2,3)) = 4$.
\end{itemize}
\end{lemma}

\begin{proof} 
Item (i) follows from the exact sequence 
\[ 0 \longrightarrow \I_{S/\PP^1 \x \PP^2} \cong \O_{\PP^1 \x \PP^2}(-2,-3) \longrightarrow \I_{C/\PP^1 \x \PP^2} \longrightarrow \I_{C/S} \cong \O_S(-H) \longrightarrow 0 \]
tensored by $\O_{\PP^1 \x \PP^2}(2,2)$ and the isomorphisms
\[ \O_S(2L-H)\cong \O_S(H-2M) \cong
\O_S(A_1+A_2),\] cf. \eqref{eq:H-2M}.
Item (ii) is an immediate consequence of (i). 

Item (iii) follows from the above sequence tensored by $\O_{\PP^1 \x \PP^2}(2,3)$
and the equality $h^0(S,R+A_1+A_2)=5$. Item (iv) follows similarly.
\end{proof}

\begin{remark} \label{rem:7ram}
 Lemma \ref{lemma:7-cohprop}(ii) is of particular interest. Indeed, it implies that the image of the moduli map $\chi_7^{ns}: \P_7^{\n,ns}\to \R_7$ lies in the ramification locus of the Prym map $\R_7\to \A_6$, cf. \cite{Be}.
\end{remark}

Theorem \ref{intro-thm:main} in genus $7$ follows by detecting the locus $\D_C$ in $|\I_{C/\PP^1 \x \PP^2}(2,3)|$ that parametrizes Nikulin surfaces of non-standard type.

\begin{thm} \label{lemma:fibre-7}
 The fibre of $\chi_7^{ns}: \F_7^{\n,ns} \to \R_7$ over $C$ is $4$-dimensional.  
\end{thm}

\begin{proof}
We consider the $5$-dimensional linear system $|\I_{C/\PP^1 \x \PP^2}(2,3)|$, cf. 
Lemma \ref{lemma:7-cohprop}(iii), along with its linear subsystem 
$
|\I_{C \cup A_1 \cup A_2/\PP^1 \x \PP^2}(2,3)| \subset |\I_{C/\PP^1 \x \PP^2}(2,3)|,
$
which has dimension $3$ and parametrizes Nikulin surfaces of non-standard type
by Lemma \ref{lemma:7-cohprop}(iv) and Proposition \ref{prop:7-tuttenik}. 

We are going to show the existence of a one-dimensional family of such linear subsystems, the  union of which is a hypersurface $\D_C$ in $|\I_{C/\PP^1 \x \PP^2}(2,3)|$ parametrizing Nikulin surfaces of non-standard type. 

Lemma \ref{lemma:7-cohprop}(i) yields that $C \subset Y \subset \PP^1 \times \PP^2$, where $Y$ is integral of bidegree $(2,2)$.
The linear system $|\O_Y(1,0)|$ is a ruling of conics on $Y$, and $A_1, A_2$ are in this ruling, since $C \subset Y$ and $A_j \cdot C = 6$.
For each $x \in \mathbb P^1$ we denote by $A_x$ the conic over the point $x$. Consider the map 
$$ p_*: |\O_Y(1,0)| \longrightarrow |\O_{\PP^2}(2)|, $$
sending $A_x$ to $p_*A_x$. Since $p_Y: Y \to \PP^2$ has degree two, the map
$p_*$ has degree one or two. As $p_*A_1 = p_*A_2 = A$, it has degree two. Hence there exists an involution
$\iota: \PP^1 \to \PP^1$ such that $p_*A_x = p_*A_{\iota(x)}$. Thus we have a fibration
\[ \D_C \longrightarrow \PP^1,\]
sending a surface $S$ to the pair of conjugated points defined by its vertical conics; in other words, the base $\PP^1$ is the quotient of $|\O_Y(1,0)|$ by the involution $\iota$ and the fiber over a point $\langle x, \iota(x) \rangle \in \PP^1$ is the $3$-dimensional linear subsystem $|\I_{A_x \cup A_{\iota(x)} \cup C/\PP^1 \x \PP^2}(2,3)|  \subset |\I_{C/\PP^1 \x \PP^2}(2,3)|$. Hence $\D_C$ is $4$-dimensional. 

It remains to show that the moduli map $m_C:\D_C\dasharrow  \F_7^{\n,ns}$ is generically finite. This easily follows since there are finitely many automorphism of $\mathbb{P}^1\times \mathbb{P}^2$ fixing $C$; indeed, any of them different from the identity would induce a non-trivial automorphism of $C$ itself.
\end{proof}

\section{The case of genus $9$} \label{sec:9}

Let $(S,M,H)$ be a general primitively polarized Nikulin surface of non-standard type of genus $9$.
Let $L=H-M$ and 
\[ R \sim \frac{1}{2} \left(H-N_1-N_2-N_3-N_4\right) \; \; \mbox{and} \; \; 
R' \sim L-R \sim \frac{1}{2} \left(H-N_5-N_6-N_7-N_8\right) \]
be as in \S \ref{sec:def}. We have
$R^2 = R'^2 = 2$ and $R \cdot R' = 4$. 
By Proposition \ref{prop:extheta}, the line bundle $L$ defines an embedding $S \subset \PP^7$ and $|R|$ and $|R'|$ are base point free linear systems whose general member is a smooth, irreducible curve of genus $2$. As in \eqref{segre}, the embeddings $C\subset S \subset \PP^7$ thus factor as
$$
S \subset \left(\PP^2 \times \PP^2\right)  \cap \mathbb{P}^7 \subset \PP^8.
$$

We may assume that the intersection
\[ T:= \left(\PP^2 \times \PP^2\right)  \cap \mathbb{P}^7\]
is transversal (cf. Remark \ref{rem:trans} and Proposition \ref{prop:tuttenik9} below)   and hence a  sextic Del Pezzo  threefold. Since $\omega_T \cong \O_T(-2,-2)$, we have, by adjunction, cf. \cite[\S 4.9]{GS}:

\begin{lemma} 
The surface $S$ is the complete intersection in $\PP^2 \times \PP^2$ of a hyperplane section and of a quadratic section defined by a quadric $Q$:
$$
S = Q \cap \PP^7 \cap \left(\PP^2 \times \PP^2\right) = Q \cap T \subset \PP^8.
$$
\end{lemma}

The first and second  projections $p'_S: S \to \PP^2$ and $p_S: S \to \PP^2$ are double coverings of $\PP^2$, contracting the  set of lines  $\lbrace N_1, \ldots , N_4 \rbrace$ and $\lbrace N_5, \ldots, N_8 \rbrace$, respectively.

The line bundle
\begin{equation} \label{eq:E} 
E:= H - N_1 - \cdots - N_8.
\end{equation}
plays a crucial role.

\begin{lemma} \label{lemma:E}
The linear system $|E|$ is an elliptic pencil on $S$. Furthermore, for any $F \in |E|$, we have: 
\begin{itemize}
  \item[(i)] The maps $p'_F:F \to \PP^2$ and $p_F:F \to \PP^2$ are double coverings onto smooth conics $A'$ and $A$, respectively;
\item[(ii)] $F = \left(A' \times A\right) \cap \PP^7 \subset \left(\PP^2 \times \PP^2\right) \cap \PP^7=T$.
\item[(iii)] The two surfaces $Y' := (A' \times \PP^2) \cap \PP^7$ and
$Y := \left(\PP^2 \times A\right) \cap \PP^7$ 
are minimal sextic scrolls (isomorphic to $\PP^1 \times \PP^1$) 
embedded in $T$ such that $F = Y' \cap Y$ and $F$ is anticanonical in $Y'$ and $Y$. Moreover,
$N_1 \cup \cdots \cup N_4 \subset Y'$ and $N_5 \cup \cdots \cup N_8 \subset Y$.
  \end{itemize}
\end{lemma}

\begin{proof}
  Using the fact that $\rk \Pic S=9$, it is easy to check that $E$ is nef and primitive, whence an elliptic pencil. Let $F \in |E|$. 
As $p_S$ has degree two, $p_F$ is either birational or of degree two onto its image. In the former case the image would be a quartic curve, as $R\cdot E=4$; however, $p$ contracts $N_i$, $i=1,2,3,4$, and $N_i \cdot E=2$, so the quartic would have four singular points, a contradiction. The same works for $p'_F$. Hence, (i) is proved. 

Letting $A=p(F)$ and $A'=p'(F)$, we have 
$$
F \subset \left(A \times A'\right) \subset \left(\PP^2 \times \PP^2\right) \cap \PP^7=T$$
Moreover, $A \times A'$ is the $2$-Veronese embedding of $\PP^1 \times \PP^1$ 
 defined by $|\O_{\PP^1 \times \PP^1}(2,2)|$. Hence
$F$ is a hyperplane section of it,  proving (ii).
Property (iii) easily follows since the projection $p'_Y:Y'\to A'$ realizes $Y'$ as the $\mathbb{P}^1$-bundle $\mathbb{P}(\O_{\PP^1}(3)\oplus \O_{\PP^1}(3))$ over  $A'\simeq\mathbb{P}^1$, and similarly for $Y$.
\end{proof}

\subsection{A rational parametrization of a double cover of $\F_9^{\n,ns}$}

Let us fix a Del Pezzo threefold $T:= \left(\PP^2 \times \PP^2\right)  \cap \mathbb{P}^7\subset \PP^8$. Since $T$ is smooth, the restriction map $\Pic (\PP^2 \times \PP^2) \to \Pic T$ is an isomorphism by the  Lefschetz Theorem, whence
$T$ contains no plane. In particular, both  projections $p'_T: T\to \PP^2$ and $p_T:T\to\PP^2$ realize $T$ as a $\PP^1$-bundle over $\PP^2$. We fix four vertical lines $N_1,\ldots,N_4$ and four horizontal lines $N_5,\ldots, N_8$ in $T$ such that the points $p'(N_1), \dots, p'(N_4)$ are in general position, and the same for $p(N_5), \dots p(N_8)$. 
 
 \begin{prop} \label{prop:tuttenik9} 
A general member of $\vert \I_{N_1\cup\cdots\cup N_8/T}(2,2) \vert$ is smooth and 
every smooth  $S \in \vert \I_{N_1\cup\cdots\cup N_8/T}(2,2) \vert$ is a non-standard Nikulin surface of genus $9$ polarized by \linebreak $\O_S(2,0)(N_5+ \cdots +N_8)$.  

Moreover, $\dim |\mathcal I_{N_1\cup\cdots\cup N_8/T}(2,2) |=3$.
\end{prop}
 
\begin{proof}  
Set
 $$
b' := \{ p'(N_1), \dots, p'(N_4)\} \; \; \mbox{and} \; \; 
b := \{p(N_5), \dots p(N_8)\}.
$$
and let
$A'$ (respectively, $A$) be any smooth conic passing through $b'$ (resp., $b$).
Define the following surfaces contained in $T$:
\begin{equation}
  \label{eq:quad}
  Y' := (A' \times \PP^2) \cap \PP^7 \in |\O_T(2,0)| \; \; \mbox{and} \; \; 
Y := \left(\PP^2 \times A\right) \cap \PP^7  \in |\O_T(0,2)|,
\end{equation}
which are minimal sextic scrolls isomorphic to $\PP^1 \times \PP^1$. One easily  verifies that $F:=Y' \cap Y$ is anticanonical in both $Y'$ and $Y$ and that
$N_1 \cup \cdots \cup N_4 \subset Y'$ and $N_5 \cup \cdots \cup N_8 \subset Y$.
More precisely,
\begin{eqnarray*}
  N_1+\cdots+N_4 & \in & |\O_{Y'}(2,0)|\cong |\O_{\PP^1 \x \PP^1}(4,0)| \\
  N_5+\cdots+N_8 & \in & |\O_{Y}(0,2)|\cong |\O_{\PP^1 \x \PP^1}(0,4)|.
\end{eqnarray*}
We have $\I_{F/Y' \cup Y} \cong \I_{F/Y'} \+ \I_{F/Y}$. Tensoring by $\O_{Y\cup Y'}(2,2)$ and using the fact that 
$F \in |\O_{Y'}(0,2)|$ and $F \in |\O_{Y}(2,0)|$ by \eqref{eq:quad}, we get
\begin{equation} \label{eq:minchia1}
\I_{F/Y' \cup Y}(2,2) \cong \O_{Y'}(2,0) \+ \O_{Y}(0,2) \cong \O_{\PP^1 \x \PP^1}(4,0) \+ \O_{\PP^1 \x \PP^1}(0,4) 
\end{equation}
We also have a short exact sequence
\begin{equation}\label{eq:minchia2}
0 \longrightarrow \I_{Y' \cup Y/T}(2,2) \cong \O_T  \longrightarrow \I_{F/T}(2,2)  \longrightarrow 
\I_{F/Y' \cup Y}(2,2) \longrightarrow 0,
\end{equation}
where the isomorphism follows as $Y' \cup Y \in |\O_T(2,2)|$ by \eqref{eq:quad}.
From \eqref{eq:minchia1} and \eqref{eq:minchia2} 
we get that $\I_{F/T}(2,2)$ is globally generated and the  restriction map of linear systems 
\[ |\I_{F/T}(2,2)| \longrightarrow \left(F+|\O_{Y'}(2,0)|\right) \x 
\left(F+|\O_{Y}(0,2)|\right)\]
is surjective. Hence, there is a  smooth $S \in |\I_{F/T}(2,2)|$ containing $N_1\cup\cdots\cup N_8$, and 
\begin{eqnarray} 
\label{eq:minchia3} S \cdot Y' & = & N_1 +\cdots+ N_4+F \in |\O_{S}(2,0)| \\
\label{eq:minchia4} S \cdot Y & = & N_5 +\cdots+ N_8+F \in |\O_{S}(0,2)|.
\end{eqnarray} 
In particular, the divisor
\[N_1+ \cdots +N_8 \in |\O_{S}(2,2)(-2F)|\]
is $2$-divisible in $\Pic S$. It is then easy to see that
$S$ is a non-standard Nikulin surface of genus $9$ polarized by $\O_S(2,0)(N_5+ \cdots +N_8)$. 

Finally, the sequence
\[ 0 \longrightarrow \I_{S/T}(2,2) \cong \O_T  \longrightarrow \I_{N_1\cup\cdots\cup N_8/T}(2,2)  \longrightarrow \I_{N_1\cup\cdots\cup N_8/S}(2,2) \cong \O_S(2F) \longrightarrow 0, \]
yields
$h^0(\I_{N_1\cup\cdots\cup N_8/T}(2,2))=4$.
\end{proof}

We obtain a nice parametrization of the moduli space  $\mathcal F^{\n, ns}_9$. We fix four vertical lines $N_1, \dots, N_4$ in $T$, and observe that  in the space of the Segre embedding one has
 $$
 \langle N_1 \cup \dots \cup N_4 \rangle = \PP^7
 $$
since $N_1,\dots, N_4$ are contained in a minimal sextic scroll $Y'\simeq \PP^1\times\PP^1\subset \PP^7$ defined as in the previous proof.  
 It is clear that, up to the action of $\Aut T$, we can choose this set of four lines up to the ordering of its elements. Since these four lines are spanning $\langle T \rangle = \PP^7$ and the automorphisms of $T$ are the automorphisms of $\PP^2\times\PP^2$ fixing this $\PP^7$,
 the stabilizer of $N_1 \cup N_2 \cup N_3 \cup N_4$ in $\Aut T$ coincides with the stabilizer in $\Aut (\PP^2 \times \PP^2)$ of the same set. Recall that
 $$
 \Aut (\PP^2 \times \PP^2) \cong \PGL(3) \x \PGL(3) \x \mathbb Z/2\mathbb Z,
 $$
 where the $\mathbb Z / 2\mathbb Z$-factor is due to the involution interchanging the two factors of $\PP^2 \times \PP^2$.
 For $i=1,\ldots,4$ we have $N_i = \lbrace o_i \rbrace  \times \ell_i$, where $o_i = p'(N_i)$ is a point and $\ell_i = p(N_i)$ is a line. The stabilizer of $N_1 \cup \dots \cup N_4$ 
 acts on the set of pairs $\lbrace (o_1, \ell_1), \dots, (o_4, \ell_4) \rbrace$. Hence the stabilizer is the diagonal embedding $S_4 \subset  S_4 \times S_4$. The
 action is the diagonal action: $\alpha(o_i, \ell_i) = (\alpha(o_i), \alpha(\ell_i))$. We define $N_{1 \dots 4} := \lbrace N_1, \ldots, N_4 \rbrace$ and choose a general set $N_{5...8} := \lbrace N_5, \ldots, N_8 \rbrace$ of four horizontal lines, or equivalently, four points in $p(T)=\PP^{2}$. Then the moduli space of pairs $(N_{1 \dots 4}, N_{5 \dots 8})$ is precisely the quotient
 $$
 (\PP^{2})^4 / S_4,
 $$
 where $S_4 \subset \Aut T$ is the previous group of automorphisms. Hence it acts as above: $\alpha(o, \ell) = (\alpha(o), \alpha(\ell))$ and $\alpha(\ell, o) = (\alpha(\ell), \alpha(o))$.
 Thus we have:
 
\begin{thm}  
The quotient $(\PP^{2})^4 / S_4$ is the $4$-symmetric product of 
$\PP^{2}$ and hence is rational. 
\end{thm}
 
For a general pair $(N_{1 2 3 4},N_{5 6 7 8})$, with $N_{1 \dots 4}$ fixed, the linear system $\vert \I_{N_1 \cup \ldots \cup N_8/T}(2,2) \vert$ defines a $\PP^3$-bundle over $(\PP^{2})^4$.
This bundle descends to $ (\PP^{2})^4/ S_4$, thus implying the following:
 
\begin{thm} The moduli space of fourtuples  $(S, M, H, N_{1234})$ is rational and a double cover of $\mathcal F^{\n, ns}_9$. \end{thm}

\subsection{The fibre of the  Prym-Nikulin map $\chi_9^{ns}$} 
Let both $(S, M,H) \in \F_9^{\n,ns}$ and $C \in |H|$ be general. Let $E$ be as in \eqref{eq:E} and recall Lemma \ref{lemma:E}. 
The genus $9$ case of  Theorem \ref{intro-thm:main} is a consequence of the next two results.

\begin{lemma} \label{lemma:id9}
We have
$$
\dim |\I_{C/T}(2,2)| = 2.
$$
In particular, $C$ is quadratically normal.
\end{lemma}

\begin{proof} 
Fix any $F\in |E|$. Since $2L \sim C+F$ and $T$ is projectively normal, the curve $C \cup F$ is the complete intersection in $T$ 
of two quadratic sections. Therefore, we have
\begin{equation}
  \label{eq:cohFC}
h^0(\I_{C \cup F/T}(2,2)) = 2 \; \; \mbox{and} \; \; h^1(\I_{C \cup F/T}(2,2)) = h^2(\I_{C \cup F/T}(2,2))=0.
\end{equation}

We consider the standard exact sequence
\begin{equation}
  \label{eq:stand9}
0 \longrightarrow  \I_{C \cup F/T}(2,2) \longrightarrow  \I_{C/T}(2,2) \oplus 
\I_{F/T}(2,2) \longrightarrow  \I_{C \cap F/T}(2,2) \longrightarrow  0.
\end{equation}

Taking cohomology in \eqref{eq:minchia1} and \eqref{eq:minchia2} 
yields 
\begin{equation}
  \label{eq:cohF}
  h^0(\I_{F/T}(2,2))=11 \; \; \mbox{and} \; \; h^1(\I_{F/T}(2,2))=h^2(\I_{F/T}(2,2))=0.
\end{equation}
This, together with the sequence
\[ 0 \longrightarrow \I_{F/T}(2,2) \longrightarrow \I_{F \cap C/T}(2,2) \longrightarrow \I_{F \cap C/F}(2,2) \cong \O_F(2L-C) \longrightarrow 0,\]
and the fact that $2L-C \sim F$ and $\O_F(F) \cong \O_F$, yields 
\begin{equation}
  \label{eq:cohFcC}
  h^0(\I_{F \cap C/T}(2,2))=h^0(\I_{F/T}(2,2))+h^0(\O_F)=12.
\end{equation}
Thus,  the cohomology of \eqref{eq:stand9} together with \eqref{eq:cohFC}, \eqref{eq:cohF} and  \eqref{eq:cohFcC} yields 
$h^0(\I_{C/T}(2,2))=3$.

The fact that $C$ is quadratically normal is easily checked.
\end{proof}

 \begin{proposition}  
A general $S' \in |\I_{C/T}(2,2)|$ defines a point of $\F_9^{\n,ns}$, and the moduli map $|\I_{C/T}(2,2)| \dasharrow \F_9^{\n,ns}$ is generically injective.
\end{proposition}

\begin{proof}  
As $S \cdot S' \sim 2L$ on $S$, we have 
\begin{equation} \label{eq:intS'S}
S' \cdot  S = F + C \in |\O_{S'}(2,2)|
\end{equation} 
for some $F \in |E|$. Let $Y'$ and $Y$ be as in Lemma \ref{lemma:E}(iii). 

Using the fact that $F$ is anticanonical on $Y'$, it is not difficult to show that
\[ S' \cdot Y' = N'_1 + \cdots + N'_4 + F \in |\O_{Y'}(2,2)| \cong |\O_{\PP^1 \x \PP^1}(6,2)|,\]
with $N'_1, \ldots, N'_4$ four disjoint lines in $|\O_{\PP^1 \x \PP^1}(1,0)|$. Similarly, one shows that
\[ S' \cdot Y = N'_5 + \cdots + N'_8 + F \in |\O_{Y}(2,2)| \cong |\O_{\PP^1 \x \PP^1}(2,6)|,\]
with $N'_5, \ldots, N'_8$ four disjoint lines in $|\O_{\PP^1 \x \PP^1}(0,1)|$. Hence $S'$ is a non-standard Nikulin surface of genus $9$ by Proposition \ref{prop:tuttenik9}. 

We now show that the moduli map $m_C: |\I_{C/T}(2,2)| \dasharrow \F_9^{\n,ns}$ is generically injective. Assume that $m_C(S')=m_C(S'')$, for  distinct $S',S''\in |\I_{C/T}(2,2)|$. Then there exists $\alpha \in \Aut (T)$ such that $\alpha(S')=S''$. In particular,  such an $\alpha$ would fix $C$ and thus  induce a non-trivial automorphism of $C$. This is a contradiction because the image of $m_9^{\n,ns}$ has dimension at least $20-2=18$, while the maximal dimension of a component of the locus in $\M_9$ of curves with a non-trivial automorphism is $2g-1=17$, cf. \cite{Co}. 
\end{proof}

\section{The case of genus $11$} \label{sec:11}

Let $(S,M,H)$ be a general primitively polarized Nikulin surface of non-standard type of genus $11$. Let $L=H-M$, then we have as in  \S \ref{sec:def}
\[ R \sim \frac{1}{2} \left(H-N_1-N_2\right) \; \; \mbox{and} \; \; R' \sim L-R \sim \frac{1}{2} \left(H-N_3-\cdots-N_8\right). \]
By Proposition \ref{prop:extheta}, the line bundle $L$ defines an embedding $S \subset \mathbb P^9$. Moreover  $\vert R \vert$ and $\vert R' \vert$ are base point free linear systems,
respectively  of dimensions $3$ and $2$, such that $R^2=4$, ${R'}^2=2$ and $R \cdot R'=5$.  The embedding $S \subset \PP^9$ factors as follows
$$
S \subset \left(\PP^2 \times \PP^3\right)\cap\PP^9  \subset \PP^{11},
$$
where the inclusion $\mathbb P^2 \times \mathbb P^3 \subset \mathbb P^{11}$ is the Segre embedding and $\mathbb P^9$ is linearly embedded. We may assume (cf. Remark \ref{rem:trans} and Proposition \ref{prop:tuttenik11} below) that 
the intersection
\[ T :=  \left(\PP^2 \times \PP^3\right)\cap\PP^9 \]
is transversal, so that $T$ is a smooth threefold with $K_T \sim \O_T(-1,-2)$. Hence, by the adjunction formula, $S$ is a divisor of type $(1,2)$ in $T$ and we can conclude as follows.

\begin{lemma} The surface $S$ belongs to $\vert -K_T \vert$ and  is a complete intersection in $\mathbb P^2 \times \mathbb P^3$ of three divisors, respectively of type
$(1,1)$, $(1,1)$ and $(1,2)$.
\end{lemma}

Let $ (x,y) := (x_0:x_1:x_2) \times (y_0:y_1:y_2:y_3)$ be coordinates on $\mathbb P^2 \times \mathbb P^3$. 
The equations of $S$ in $\PP^2 \times \PP^3$ can be written as
\[  a_0x_0 + a_1x_1 + a_2x_2 = b_0x_0 + b_1x_1 + b_2x_2 = c_0x_0 + c_1x_1 + c_2x_2 = 0, \]
where for $i=0,1,2$ the coefficients $a_i$ and $b_i$ are  linear forms 
while the $c_i$ are quadratic forms  in $(y_0:y_1:y_2:y_3)$. 
The equations of $T$ are 
\[a_0x_0 + a_1x_1 + a_2x_2 = b_0x_0 + b_1x_1 + b_2x_2 = 0. \]
The morphism $p_T: T \to \PP^3$ is birational and its inverse is described by
\[(y) \mapsto (a_1b_2 - a_2b_1, a_2b_0 - a_0b_2, a_0b_1 - a_1b_0) \times (y_0:y_1:y_2:y_3).\]
Equivalently, $p_T$ is the blow-up of the scheme $\gamma$ defined by the $2 \times 2$ minors of  
\[ \left( \begin{array}{ccc}
a_0 & a_1 & a_2\\
b_0 & b_1 & b_2 \end{array} \right).\] 
Since $T$ is smooth, $\gamma$ is
a smooth (rational normal cubic) curve. Let $P_{\gamma}:=p_T^{-1}(\gamma)$ be the exceptional divisor of $p_T$.

\begin{lemma} \label{lemma:eccdiv}
We have $P_{\gamma} \in \vert \O_T(-1,2) \vert$ and $P_{\gamma} \cong \PP^1 \times \PP^1$. Under this identification, $\O_{P_{\gamma}}(0,1)\cong \O_{\PP^1 \times \PP^1}(0,3)$ and $\O_{P_{\gamma}}(1,0) \cong \O_{\PP^1 \times \PP^1}(1,1)$.
\end{lemma}

\begin{proof} 
 We have 
\[ 
\O_T(P_{\gamma}) \cong \omega_T \* p_T^*(\omega_{\PP^3}^{\vee}) \cong \O_T(-1,-2) \* 
\O_T(0,4) \cong \O_T(-1,2).\]
As is well known,  $\N_{\gamma/\PP^3} \cong \O_{\PP^1}(5) \oplus \O_{\PP^1}(5)$, whence $P_{\gamma} \cong \PP^1 \times \PP^1$. Since $\gamma \subset \PP^3$ is a curve of degree $3$, it follows that $\O_{P_{\gamma}}(0,1)\cong \O_{\PP^1 \times \PP^1}(0,3)$.
 Finally,  we have
\[ \O_{\PP^1 \x \PP^1}(-2,-2) \cong \omega_{P_{\gamma}} \cong \O_{P_{\gamma}}(K_T+P_{\gamma}) \cong \O_{P_{\gamma}}(-2,0),\]
whence $\O_{P_{\gamma}}(1,0) \cong \O_{\PP^1 \times \PP^1}(1,1)$.
\end{proof}

\begin{lemma} \label{lemma:intSPg}
We have
\[ S \cdot P_{\gamma} = \Gamma + N_3 + \dots + N_8, \]
where $\Gamma$ is a smooth element of  $\vert \O_{P_{\gamma}}(1,0)\vert=\vert \mathcal O_{\mathbb P^1 \times \mathbb P^1}(1,1) \vert$. In particular, $p'_{\Gamma}$ is a two to one map onto a line. 

Moreover, $\Gamma$ has the following properties:

\begin{itemize}
\item[(i)]   $\Gamma\cdot N_3 = \dots = \Gamma \cdot N_8 = 1$ and $\Gamma \cdot N_1 = \Gamma \cdot N_2 = 2$.
\item[(ii)]   $\Gamma + N_1 + N_2 \sim R'$.
\end{itemize}
\end{lemma}

\begin{proof} We know that $N_3, \ldots, N_8$ are contracted by 
$p_S$, whence they are six disjoint fibres of $p_{P_{\gamma}}: P_{\gamma} \to \gamma$. On the other hand, $S \in \vert \mathcal O_T(1,2) \vert$, hence its restriction to
$P_{\gamma}$ belongs to $\vert \mathcal O_{\mathbb P^1 \times \mathbb P^1}(1,7)\vert$ by Lemma \ref{lemma:eccdiv}. This implies that $\Gamma \in \vert \mathcal O_{\mathbb P^1 \times \mathbb P^1}(1,1) \vert=\vert \O_{P_{\gamma}}(1,0)\vert$, and it immediately follows that $p'$ maps $\Gamma$ two to one onto a line. 
If $\Gamma$ is not smooth, then it contains a fibre $N_9$ of $p_{P_{\gamma}}$. But then one can check (on $S$) that $N_9$ is orthogonal to $R, N_1, \dots, N_8$. Hence
$\Pic S$ has rank $\geq 10$, against the generality of $S$.
The properties (i) and (ii) are easy to check. 
\end{proof}

Consider the line $\ell := p'(\Gamma)$ and the surface 
\begin{equation} \label{eq:defPl}
P_{\ell} := {p'}^{-1}(\ell) \cap T \in |\O_T(1,0)|.
\end{equation}
 Let $l_0x_0 + l_1x_1 + l_2x_2 = 0$ be the equation of $\ell$, with $l_0,l_1,l_2 \in \CC$. Then $P_{\ell}$ is defined by
\[ l_0x_0 + l_1x_1 + l_2x_2 = a_0x_0 + a_1x_1 + a_2x_2 = b_0x_0 + b_1x_1 + b_2x_2 = 0. \]
 The surface $P_{\ell}$ is a $\PP^1$-bundle over $\ell$ and $p(P_{\ell}) \subset \PP^3$ is a quadric through $\gamma$ defined by the equation 
\[\det  \left( \begin{array}{ccc}
l_0 & l_1 & l_2\\
a_0 & a_1 & a_2\\
b_0 & b_1 & b_2  \end{array} \right) = 0. \] 

\begin{lemma} \label{lemma:Pl}
One has  
\[S \cdot P_{\ell} = \Gamma + N_1 + N_2. \] 
Moreover, $p(P_{\ell})$ is smooth and $P_{\ell} \cong \PP^1 \times \PP^1$, with
$\O_{P_{\ell}}(1,1) \cong  \O_{\PP^1 \times \PP^1}(2,1)$.
\end{lemma}

\begin{proof} 
The first assertion follows from Lemma \ref{lemma:intSPg}(ii) and \eqref{eq:defPl}. Next assume $p(P_{\ell})$ is singular. Then it is a rank $3$ cone of vertex $e =p(N_1) \cap p(N_2)$, and $e \in \gamma$. 
But then the curve $p_T^{-1}(e)$ is contained in $S \cap P_{\ell}$ as a proper component of $\Gamma$, against the irreducibility of $\Gamma$.
Finally, since $p(P_{\ell})$ is a smooth quadric, we have $\mathcal O_{P_{\ell}}(0,1) \cong \mathcal O_{\mathbb P^1 \times \mathbb P^1}(1,1)$. Hence, the isomorphism  $\mathcal O_{P_{\ell}}(1,1) \cong  \mathcal O_{\mathbb P^1 \times \mathbb P^1}(2,1)$ follows. 
\end{proof}

In the considerations so far, $\gamma$, $T$ and $P_{\gamma}$ are fixed and independent of $S$, whereas $\Gamma$ depends on $S$ and determines the line $\ell \subset \PP^2$ and thus the surface $P_{\ell}$. Actually, $\ell$ alone determines both $P_{\ell}$ and $\Gamma$, as $P_{\ell}={p'}_T^{-1}(\ell)$ and
$\Gamma= P_{\ell} \cap P_{\gamma}$. In order to parametrize all Nikulin surfaces we will indeed let $\ell \subset \PP^2$ vary.

\subsection{Rationality of $\F_{11}^{\n,ns}$} \label{ss:ratpar11}

Fix any smooth rational normal cubic curve $\gamma \subset \PP^3$ and let $p_T:T \to \PP^3$ be the blow-up along $\gamma$ with exceptional divisor $P_{\gamma}$. Then $T \subset \PP^2 \x \PP^3$ and we denote as before by $p'_T: T \to \PP^2$ the first projection. Any line $\ell \subset \PP^2$ determines a surface $P_{\ell}:={p'}^{-1}(\ell) \cap T \in |\O_T(1,0)|$ and a curve $\Gamma_{\ell}:= P_{\ell} \cap P_{\gamma} \in   |\O_{\PP^1 \x \PP^1}(1,1)|$, which is smooth for general $\ell$. 

\begin{prop} \label{prop:tuttenik11} 
Let $\ell$ be general. Then a general member of $\vert \I_{\Gamma_{\ell}/T}(1,2) \vert$ is smooth and every smooth $S \in \vert \I_{\Gamma_{\ell}/T}(1,2) \vert$ is a non-standard Nikulin surface of genus $11$ polarized by $\mathcal O_{S}(1,2)(-\Gamma_{\ell})$.

Moreover, 
$\dim |\I_{\Gamma_{\ell}/T}(1,2)|=12$.
\end{prop}

\begin{proof} 
Consider the exact sequences of ideal sheaves
\begin{equation} \label{eq:id1}
\xymatrix{0 \ar[r] & \I_{P_{\gamma}/T}(1,2) \ar[r] & \I_{\Gamma_{\ell}/T}(1,2) \ar[r] & \O_{P_{\gamma}}(1,2)(-\Gamma_{\ell}) \ar[r] &  0}
\end{equation}
and
\begin{equation} \label{eq:id2}
\xymatrix{0 \ar[r] & \I_{P_{\ell}/T}(1,2) \ar[r] & \I_{\Gamma_{\ell}/T}(1,2) \ar[r] & \O_{P_{\ell}}(1,2)(-\Gamma_{\ell}) \ar[r] &  0.}
\end{equation}
By \eqref{eq:defPl} and Lemma \ref{lemma:eccdiv} we have
\begin{equation}
  \label{eq:id3}
 \I_{P_{\gamma}/T}(1,2) \cong \O_T(2,0) \; \; \mbox{and} \; \;  \I_{P_{\ell}/T}(1,2) \cong \O_T(0,2), 
\end{equation}
and by Lemmas \ref{lemma:eccdiv}, \ref{lemma:intSPg} and \ref{lemma:Pl} we have
\begin{eqnarray}
  \label{eq:id4}
 \O_{P_{\gamma}}(1,2)(-\Gamma_{\ell}) & \cong \; \; \O_{P_{\gamma}}(0,2) & \cong \; \; \O_{\mathbb P^1 \times \mathbb P^1}(0,6) \\ 
\label{eq:id5}
\O_{P_{\ell}}(1,2)(-\Gamma_{\ell}) & \cong \; \;\O_{P_{\ell}}(2,0) & \cong \; \; \O_{\mathbb P^1 \times \mathbb P^1}(2,0).  
\end{eqnarray}
Thus, either of \eqref{eq:id1} and \eqref{eq:id2} shows that $\I_{\Gamma_{\ell}/T}(1,2)$ is globally generated. In particular, a general $S \in \vert \mathcal I_{\Gamma_{\ell} / T}(1,2) \vert$ is smooth and hence a $K3$ surface by  adjunction. 

From \eqref{eq:id1}-\eqref{eq:id5} one obtains that $h^0(\mathcal I_{\Gamma_{\ell}/T}(1,2)) = 13$ and that the restriction maps
\begin{eqnarray*}
\rho_{\gamma}: \vert \mathcal I_{\Gamma_{\ell}/T}(1,2) \vert & \longrightarrow \; \;  \Gamma_{\ell} + \vert \mathcal O_{P_{\gamma}}(0,2) \vert  & =  \; \; \Gamma_{\ell} + \vert \O_{\mathbb P^1 \times \mathbb P^1}(0,6)\vert  \\
 \rho_{\ell}:  \vert \mathcal I_{\Gamma_{\ell}/T}(1,2) \vert & \longrightarrow \; \;  \Gamma_{\ell} + \vert \mathcal O_{P_{\ell}}(2,0) \vert & = \; \;  \Gamma_{\ell} + \vert \O_{\mathbb P^1 \times \mathbb P^1}(2,0)\vert 
\end{eqnarray*}
are surjective. A general member of  $\vert \mathcal O_{P_{\gamma}}(0,2) \vert$  and of $\vert \mathcal O_{P_{\ell}}(2,0) \vert$ consists of $6$ and $2$
disjoint lines, respectively. Hence a general $S \in \vert \mathcal I_{\Gamma_{\ell} / T}(1,2) \vert$ contains a configuration of $8$ disjoint lines, say $N_1, \ldots,  N_8$, such that
\begin{equation} \label{eq:dupalle}
\Gamma_{\ell} + N_1 + N_2 = S \cdot P_{\ell} \in |\O_{S}(1,0)| \; \; \mbox{and} \; \; \Gamma_{\ell} + N_3 + \dots + N_8 = S \cdot P_{\gamma} \in |\O_{S}(-1,2)|
\end{equation}
(using \eqref{eq:defPl} and Lemma \ref{lemma:eccdiv}). 
By \eqref{eq:dupalle}, we also get
 \[ 2\Gamma_{\ell} + N_1 + \dots + N_8 \in |\O_{S}(0,2)|,\]
whence $N_1+ \dots +N_8$ is divisible by $2$ in $\Pic S$. One easily checks that 
\[ \O_{S}(1,2)(-\Gamma_{\ell}) \sim \O_{S}(0,2) + N_1 + N_2 \sim  \O_{S}(2,0) + N_3 + \dots + N_8\]
is a genus $11$ polarization having zero intersection with all $N_1,\ldots, N_8$. The fact that $S$ is of non-standard type is an immediate consequence of \eqref{eq:dupalle}. 
\end{proof}

By the considerations at the beginning of the section, any smooth genus $11$ Nikulin surface of nonstandard type is an element of $|\O_T(1,2)|$ and defines a smooth $\Gamma_{\ell}$ mapping $2:1$ to a line $\ell$ on $\PP^2$ under $p$. It moreover comes equipped with $6$ horizontal rational curves $N_3 \cup \cdots  \cup N_8$, and thus determines $6$ points on $\gamma$. 

\begin{lemma} \label{lemma:tuttenik11-2}
  Fix a general line $\ell \subset \PP^2$ and six general points $p_3,\ldots,p_8$ on $\gamma$. Let $N_i=P_\gamma\cap p_T^{-1}(p_i)$, $i=3,\ldots,8$. Then
$\dim |\I_{\Gamma_{\ell}+N_3+\cdots +N_8/T}(1,2)|=6$.
\end{lemma}

\begin{proof}
The statement follows from the ideal sequence
\begin{equation} \label{eq:id0}
0 \longrightarrow \I_{P_{\gamma}/T}(1,2) \longrightarrow \I_{\Gamma_{\ell}+N_3+\cdots +N_8/T}(1,2) \longrightarrow  \I_{\Gamma_{\ell}+N_3+\cdots +N_8/P_{\gamma}}(1,2) \longrightarrow  0,
\end{equation}
along with \eqref{eq:id3} and the fact that $\I_{\Gamma_{\ell}+N_3+\cdots +N_8/P_{\gamma}}(1,2) \cong \O_{P_{\gamma}}$ by Lemma \ref{lemma:intSPg}. 
\end{proof}

We consider the $\PP^6$-bundle $\P$ over $(\PP^2)^\vee\times\Sym^6(\gamma)$, whose fiber over the point \linebreak $(\ell, p_3+\cdots+p_8)$ is the linear system $|\I_{\Gamma_{\ell}+N_3+\cdots +N_8/T}(1,2)|$ with $N_i=P_\gamma\cap p_T^{-1}(p_i)$.  Our construction provides a dominant rational moduli map
$$\xymatrix{f:\P \ar@{-->}[r] &  \F_{11}^{\n,ns},}$$
and the fibers are orbits of the group of automorphisms of $T$ that fix the exceptional divisor $P_\gamma$, namely, of the group of automorphisms of $\gamma\subset \PP^3$. In particular $\F_{11}^{\n,ns}$ is birational to $\P/\Aut(\gamma)$.

\begin{thm}
The moduli space $\F_{11}^{\n,ns}$ is rational.
\end{thm}
\begin{proof}
Since there are no non-trivial automorphisms of $\PP^1$ mapping a set of $6$ general points to itself, $\P/\Aut(\gamma)$ is birational to a $\PP^6$-bundle over $(\PP^2)^\vee\times\left(\Sym^6(\gamma))/\Aut(\gamma)\right)$. It is then enough to recall that $\Sym^6(\PP^1))/\Aut(\PP^1)$ is birational to the moduli space $\M_2$ of genus $2$ curves, which is known to be rational, cf.  \cite{Ig}.
\end{proof}

\subsection{The fibre of the  Prym-Nikulin map $\chi_{11}^{ns}$} 

The genus $11$ case of  Theorem \ref{intro-thm:main} is a consequence of 
the following:

\begin{lemma} \label{lemma:base}
  Let $(S,M,H)$ be a general member of $\F_{11}^{\n,ns}$. For any $C \in |H|$, the linear system $|\I_{C/T}(1,2)|$ is a pencil of nonisomorphic non-standard Nikulin surfaces of genus $11$. 
\end{lemma}

\begin{proof}
  The ideal sequence of $C \subset S \subset T$ twisted by $\O_T(1,2)$ becomes
\begin{equation} \label{eq:ids1}
0 \longrightarrow \O_T   \longrightarrow \I_{C/T}(1,2) \longrightarrow \O_S(\Gamma) \longrightarrow 0,
\end{equation}
by Proposition \ref{prop:tuttenik11}. As a consequence, the $1$-dimensional linear system  $|\I_{C/T}(1,2)|$ contains $C\cup \Gamma$ as its base locus and thus parametrizes Nikulin surfaces again by Proposition \ref{prop:tuttenik11}. Let $S',S''\in |\I_{C/T}(1,2)|$ be two distinct points parametrizing isomorphic Nikulin surfaces.  Then there exists $\alpha \in \Aut (T)$ such that $\alpha(S') = S''$, $\alpha(\Gamma) = \Gamma$ and $\alpha(C)=C$. In particular, such an $\alpha$ would induce a non-trivial automorphism of $C$. Note that the image of $m_{11}^{\n,ns}$ has dimension at least $22-1=21$, which is an upper bound for the dimension of any component of the locus in $\M_{11}$ of curves with a non-trivial automorphism, cf. \cite{Co}.  However, this bound is reached only by the hyperelliptic locus and $[C]$ does not lie in it as its Clifford index is $4$ by \cite[Prop. 2.3]{KLV}.\end{proof}

  \end{document}